\documentclass{amsart}
\usepackage{amsfonts}

\usepackage{eurosym}
\usepackage{amsmath}
\usepackage{amssymb}
\usepackage[all]{xy}
\usepackage{tikz-cd}
\usepackage{amsthm}
\usepackage[square,sort,comma,numbers]{natbib}
\usepackage{hyperref}

\theoremstyle{plain}
\newtheorem{thm}{Theorem}[section]
\newtheorem{lem}{Lemma}[section]
\theoremstyle{definition}
\newtheorem{defn}{Definition}[section]
\newtheorem*{defn*}{Definition}

\begin{document}
\title[The Fra\"iss\'e limit of matrix algebras with the
rank metric]{The Fra\"iss\'e limit of matrix algebras
with the rank metric}
\author{Aaron Anderson}
\address{Mathematics Department\\
California Institute of Technology\\
1200 E. California Blvd\\
MC 253-37\\
Pasadena, CA 91125}
\email{awanders@caltech.edu}
\urladdr{https://www.math.ucla.edu/~aaronanderson/}
\thanks{A.A. was supported by Caltech's Summer Undergraduate Research Fellowships (SURF) program and by a Rose Hills Summer Undergraduate Research Fellowship. We would also like to thank Martino Lupini for mentorship and guidance throughout this project.}
\dedicatory{}

\begin{abstract}
We realize the $\mathbb{F}_q$-algebra $M(\mathbb{F}_q)$ studied by von
Neumann and Halperin as the Fra\"iss\'e limit of the class of
finite-dimensional matrix algebras over a finite field $\mathbb{F}_q$
equipped with the rank metric. We then provide a new Fra\"iss\'e-theoretic
proof of uniqueness of such an object. Using the results of Carderi and
Thom, we show that the automorphism group of $\mathrm{Aut}(\mathbb{F}_q )$
is extremely amenable. We deduce a Ramsey-theoretic property for the class
of algebras $M(\mathbb{F}_q)$, and provide an explicit bound for the
quantities involved. 
\end{abstract}

\maketitle
\section{Introduction}

Fix $q$, a prime power. Let $\mathcal{K}(\mathbb{F}_q)=\{M_n(\mathbb{F}_q):n
\in \mathbb{N}\}$. If $m | n$, then let $\iota_{n,m}$ be the embedding of $%
M_m(\mathbb{F}_q)$ into $M_n(\mathbb{F}_q)$, given by $\iota_{n,m}(x)=x
\otimes 1_{n/m}$. If $n_0,n_1,\dots$ satisfies $n_k|n_{k+1}$ and $\lim_{k
\to \infty} n_k=\infty$, then we call $n_0,n_1,\dots$ a \textit{factor
sequence}, and 
\begin{equation*}
M_{n_0}(\mathbb{F}_q) \overset{\iota_{n_1,n_0}}{\hookrightarrow }M_{n_1}(%
\mathbb{F}_q) \overset{\iota_{n_2,n_1}}{\hookrightarrow }\dots
\end{equation*}
is an inductive sequence of $\mathbb{F}_q$-algebras, so it has a direct
limit, which we call $M_0(\mathbb{F}_q)$. Let $\iota_{n_k}$ be the
corresponding inclusion of $M_{n_k}(\mathbb{F_q})$ into $M_0(\mathbb{F}_q)$.
On each $M_n(\mathbb{F}_q)$, we can define a metric, $d(x,y)=\frac{\mathrm{%
rank}(x-y)}{n}$. Under these metrics, each inclusion $\iota_{n,m}$ is an
isometry, so a metric is induced on $M_0(\mathbb{F}_q)$. Let $M(\mathbb{F}_q)
$ be the completion of $M_0(\mathbb{F}_q)$ under this metric, which is also
an $\mathbb{F}_q$-algebra. In a manuscript eventually reworked and published
by his student Halperin \cite{Halperin3}, von Neumann showed that $M(\mathbb{%
F}_q)$ is uniquely defined, that is, it does not depend on the choice of
factor sequence.

In classical model theory, a Fra\"iss\'e class $\mathcal{K}$ is a collection
of finitely-generated structures (or isomorphism classes thereof) in a given
language, satisfying a few additional properties, which guarantee the
existence and uniqueness of a Fra\"iss\'e limit associated with the class.
The Fra\"iss\'e limit is a countably-generated structure $F$ such that for
any structure $A \in \mathcal{K}$, any isomorphism between substructures $%
A,B\in \mathcal{K}$ of $F$ can be extended to an automorphism of $F$, a
property known as $\mathcal{K}$-homogeneity \cite{Hodges}. This theory
carries over to model theory of metric structures, where the limit need only
be approximately $\mathcal{K}$-homogeneous \cite{begnac}. In Section~\ref%
{sec:class}, we show that in the language of $\mathbb{F}_q$-algebras with a
metric, the class $\mathcal{K}(\mathbb{F}_q)$ is a Fra\"iss\'e class, and in
Section~\ref{sec:limit}, we show that $M(\mathbb{F}_q)$ is its Fra\"iss\'e
limit. We also use give a direct Fra\"iss\'e-theoretic proof of the
uniqueness of $M(\mathbb{F}_q)$, originally established by von Neumann with
a different argument. 

Because of the added structure of the rank metric, there only exists an
embedding from $M_{m}(\mathbb{F}_{q})$ into $M_{n}(\mathbb{F}_{q})$ if $m$
divides $n$. However, in Sections \ref{sec:limit} and \ref{sec:exam}, we
wish to approximate embeddings from $M_{m}(\mathbb{F}_{q})$ into $M(\mathbb{F%
}_{q})$ by embeddings from $M_{m}(\mathbb{F}_{q})$ into some finite stage of
the direct limit defining $M(\mathbb{F}_{q})$, say, $M_{n_{k}}(\mathbb{F}%
_{q})$ for some $k$. To do this, unless some $n_{k}$ is a multiple of $m$,
we must instead approximate the embedding with approximate embeddings. In
Section \ref{sec:lem} we consider a natural notion of approximate embedding,
and show that any embedding  into $M(\mathbb{F}_{q})$ can be approximated
arbitrarily well by these approximate embeddings into finite stages of the
limit. In order to establish this fact, we consider a presentation of $M_{n}(%
\mathbb{F}_{q})$ in terms of generators and relations studied by Kassabov in 
\cite{matrixgen}, and prove that the defining relations are \emph{stable }%
with respect to the rank metric. Stability problems for relations in metric
groups and operator algebras have been estensively studied, also due to
their connections with notions such as (linear) soficity and hyperlinearity
in group theory, and (weak) semiprojectivity and $\mathcal{R}$-embeddability
in operator algebras.

The Kechris-Pestov-Todorcevic correspondence establishes an equivalence
between a \textit{Ramsey property} of a Fra\"iss\'e class and \textit{%
extreme amenability} of the automorphism group of its Fra\"iss\'e limit. The
Ramsey property is a generalization of the Ramsey theorem, reducing to the
standard Ramsey theorem for the Fra\"iss\'e class of finite linear
orderings. Extreme amenability pertains to the topological dynamics of the
group: a topological group $G$ is extremely amenable when any continuous
action of $G$ on a compact Hausdorff space $X$ leaves some point fixed \cite%
{kpt}. This too carries over to metric Fra\"iss\'e structures, but again,
the Ramsey property is only approximate \cite{melleraytsankov}. In Section~%
\ref{sec:exam}, we reduce the extreme amenability of $\mathrm{Aut}(M(\mathbb{%
F}_q))$ to the extreme amenability of the unit group of $M(\mathbb{F}_q)$,
proven by Carderi and Thom \cite{carderithom}.

It seems worth mentioning that the study of natural limiting objects of
finite-dimensional matrix algebras has also connections with  computer
science and applied graph theory. In \cite{fractalgraphs}, the authors study
Kronecker graphs, which are constructed by taking repeated tensor products
of the adjacency matrices of graphs. By taking the tensor product
sufficiently many times, one can construct a graph which is approximately
self-similar, a process suitable for modelling fractal structures which
appear in nature, or graphs such as social networks. However, in order to
create a graph which would have genuine fractal structure, one would need to
take the tensor product of an infinite sequence of matrices, which would no
longer be a well-defined matrix. Such an object does however exist, as the
limit of a Cauchy sequence of partial products, in the algebra $M(\mathbb{F}%
_{q})$, so it may be possible to gain new insight into fractal graphs by
studying this algebra further.

\section{$\mathcal{K}(\mathbb{F}_q)$ is a Fra\"iss\'e class}

\label{sec:class}

\begin{defn}
Let $\mathcal{K}$ be a class of finitely-generated metric structures in a
particular language. $\mathcal{K}$ is a metric Fra\"iss\'e class \cite{Eagle}
if and only if the following properties are satisfied: 

\begin{itemize}
\item Joint Embedding Property (JEP): For any $A,B \in \mathcal{K}$, there
exists some $C \in \mathcal{K}$ such that $A$ and $B$ both embed into $C$. 

\item Near Amalgamation Property (NAP): For any $A,B_0,B_1 \in \mathcal{K}$,
embeddings $\phi_i: A \hookrightarrow B_i$, and $\varepsilon>0$, there
exists some $C \in \mathcal{K}$ with embeddings $\psi_i: B_i \hookrightarrow
C$ such that $d(\psi_0 \circ \phi_0,\psi_1 \circ \phi_1)<\varepsilon$. 

\item Weak Polish Property (WPP): For any class satisfying JEP and NAP, we
can define, for each $n \in \mathbb{N}$, a class $\mathcal{K}_n$ of
structures in $\mathcal{K}$, with specified generating tuples of size at
most $n$. We then define a pseudometric on $\mathcal{K}_n$ (relying on JEP
and NAP) by  
\begin{equation*}
d_n(\bar a,\bar b)=\inf d(\phi(\bar a),\psi(\bar b))
\end{equation*}
where $\phi:\langle \bar a\rangle \hookrightarrow C, \psi:\langle \bar b
\rangle \hookrightarrow C$ are embeddings into the same structure $C \in 
\mathcal{K}$. The WPP is true when each of these pseudometrics $d_n$ is
separable. 
\end{itemize}
\end{defn}

We will now verify that $\mathcal{K}(\mathbb{F}_q)$ satisfies each of these
properties, and is thus a Fra\"iss\'e class, implying the existence of a
unique Fra\"iss\'e limit.

\paragraph{Joint Embedding Property}

If $m | n$, then let $\iota_{n,m}$ be the embedding of $M_m(\mathbb{F}_q)$
into $M_n(\mathbb{F}_q)$, given by $\iota_{n,m}(x)=x \otimes 1_{n/m}$ \cite%
{Halperin2}.

Let $A$ and $B$ be structures in $\mathcal{K}(\mathbb{F}_q)$, that is, $%
A=M_a(\mathbb{F}_q)$ and $B=M_b(\mathbb{F}_q)$. Then if $C=M_{ab}(\mathbb{F}%
_q)$, there exists an embedding $\iota_{ab,a}: A \hookrightarrow C$, and an
embedding $\iota_{ab,b}: B \hookrightarrow C$, so $A$ and $B$ can be jointly
embedded into $C$.

\paragraph{Amalgamation Property}

In this case, $\mathcal{K}(\mathbb{F}_q)$ satisfies not only the Near
Amalgamation Property, but the same amalgamation property as discrete
structures, allowing us to dispense with the $\varepsilon$: for any $%
A,B_0,B_1 \in \mathcal{K}$, embeddings $\phi_i: A \hookrightarrow B_i$,
there exists some $C \in \mathcal{K}$ with embeddings $\psi_i: B_i
\hookrightarrow C$ such that $\psi_0 \circ \phi_0=\psi_1 \circ \phi_1$.

Let $A,B_0,B_1$ be structures in $\mathcal{K}(\mathbb{F}_q)$, with
embeddings $\phi_i:A \to B_i$. Let $A=M_a(\mathbb{F}_q)$ and $B_i=M_{b_i}(%
\mathbb{F}_q)$. As $A,B_i$ are matrix algebras over $\mathbb{F}_q$, and thus
finite-dimensional central simple algebras over $\mathbb{F}_q$, the
Skolem-Noether Theorem \cite{Farb} shows that each embedding $\phi_i: A
\hookrightarrow B_i$ must be a composition of $\iota_{b_i,a}$ with an inner
automorphism of $B_i$, given by conjugating by some unit $y_i \in B_i^*$.
Thus we may assume without loss of generality that each $\phi_i=\iota_{b_i,a}
$.

Thus if $C=M_{c}(\mathbb{F}_q)$, where $b_0,b_1$ both divide $c$, we can use
the automorphisms $\iota_{c,b_i}$ to make the following diagram commute: 
\begin{equation*}
\xymatrix{
	 & M_c(\mathbb{F}_q)& \\
	M_{b_0}(\mathbb{F}_q)\ar[ur]^{\iota_{c,b_0}} & & M_{b_1}(\mathbb{F}_q)\ar[ul]_{\iota_{c,b_1}} \\
	 & M_a(\mathbb{F}_q)\ar[ul]^{\iota_{b_0,a}} \ar[ur]_{\iota_{b_1,a}}&
}
\end{equation*}
This commutes because for any $i,j,k$, 
\begin{equation*}
\iota_{ijk,ij}\circ \iota_{ij,i}(x)=\iota_{ij,i}(x)\otimes 1_k=x \otimes
1_j\otimes 1_k = x \otimes 1_{jk}=\iota_{ijk,i}(x)
\end{equation*}
and thus $\iota_{b_0b_1,b_i} \circ \iota_{b_i,a} = \iota_{b_0b_1,a}$ for
each $i$.

\paragraph{Weak Polish Property}

There are only countably many structures in $\mathcal{K}(\mathbb{F}_q)$, and
each one is finite. Thus each $\mathcal{K}(\mathbb{F}_q)_n$ is countable,
and thus trivially separable.

\section{Stability Lemma}

\label{sec:lem}

\subsection{$\protect\delta$-embeddings}

\begin{defn}
Define a (not necessarily unital) homomorphism $\phi: M_m(\mathbb{F}_q) \to
M_n(\mathbb{F}_q)$ to be a \textit{$\delta$-embedding} when there exists
some unit $y \in M_n(\mathbb{F}_q)$ and some number $k$ such that for each $%
x \in M_m(\mathbb{F}_q)$, $\phi(x)=y(x^{\oplus k}\oplus 0^{n-mk})y^{-1}$,
and $\frac{mk}{n}\geq 1-\delta$.
\end{defn}

$\delta$-embeddings will be used as a proxy for actual embeddings, because
we cannot always guarantee that there will be an embedding $M_m(\mathbb{F}%
_q) \hookrightarrow M_n(\mathbb{F}_q)$, unless we know that $m$ divides $n$.
We can reconstruct an actual embedding by taking a limit of $\delta$%
-embeddings with $\delta$ decreasing to 0.

\subsection{Proving the Lemma}

\begin{lem}
\label{stablem}  Let $M(\mathbb{F}_q)$ be the completion of the direct limit
of the sequence $M_{n_0}(\mathbb{F}_q) \hookrightarrow M_{n_1}(\mathbb{F}%
_q)\hookrightarrow\dots$, and let $\phi: M_n(\mathbb{F}_q) \hookrightarrow M(%
\mathbb{F}_q)$ be an embedding. Then for each $\varepsilon, \delta>0$, there
exists some $N$ such that if $K\geq N$, if $n_K=mn+r$ with $0 \leq r < n$,
then there is a $\frac{r}{n_K}$-embedding $\psi: M_n(\mathbb{F}_q) \to
M_{n_K}(\mathbb{F}_q)$ such that $d(\iota_{n_K} \circ \psi,\phi)<\varepsilon$%
. In particular, if $n$ divides $n_K$, then $\psi$ is an embedding.
\end{lem}

\begin{proof}
Fix $\varepsilon >0$. As established in \cite{matrixgen}, for a prime $p$, $%
M_n(\mathbb{F}_p)$ is the ring presented by the following generators and
relations 
\begin{equation*}
M_n(\mathbb{F}_p)= \langle a,b|a^n=b^n=0,ba+(p+1)a^{n-1}b^{n-1}=1\rangle
\end{equation*}
where $a$ and $b$ are the off-diagonal matrices

\begin{equation*}
\begin{bmatrix}
0 & \cdots & \cdots & \cdots & \cdots & \cdots & \cdots & 0 \\ 
1 & 0 &  &  &  &  &  & \vdots \\ 
0 & 1 & 0 &  &  &  &  & \vdots \\ 
\vdots & \ddots & \ddots & \ddots &  &  &  & \vdots \\ 
\vdots &  & \ddots & \ddots & \ddots &  &  & \vdots \\ 
\vdots &  &  & \ddots & 1 & 0 &  & \vdots \\ 
\vdots &  &  &  & \ddots & 1 & 0 & \vdots \\ 
0 & \cdots & \cdots & \cdots & \cdots & 0 & 1 & 0 \\ 
&  &  &  &  &  &  & 
\end{bmatrix}%
, 
\begin{bmatrix}
0 & 1 & 0 & \cdots & \cdots & \cdots & \cdots & 0 \\ 
\vdots & 0 & 1 & \ddots &  &  &  & \vdots \\ 
\vdots &  & 0 & 1 & \ddots &  &  & \vdots \\ 
\vdots &  &  & \ddots & \ddots & \ddots &  & \vdots \\ 
\vdots &  &  &  & \ddots & \ddots & \ddots & \vdots \\ 
\vdots &  &  &  &  & 0 & 1 & 0 \\ 
\vdots &  &  &  &  &  & 0 & 1 \\ 
0 & \cdots & \cdots & \cdots & \cdots & \cdots & \cdots & 0 \\ 
&  &  &  &  &  &  & 
\end{bmatrix}%
\end{equation*}

Choose $p$ to be the prime such that $p^k=q$. The set of matrices in $M_n(%
\mathbb{F}_q)$ with coordinates from $\mathbb{F}_p$ is an isomorphic copy of 
$M_n(\mathbb{F}_p)$, so we can find $a,b \in M_n(\mathbb{F}_q)$ that
generate the embedded copy of $M_n(\mathbb{F}_p)$ as a ring. As the embedded
copy of $M_n(\mathbb{F}_p)$ contains a basis of $M_n(\mathbb{F}_q)$ as a $%
\mathbb{F}_q$-vector space, $a,b$ generate $M_n(\mathbb{F}_q)$ as a $\mathbb{%
F}_q$-algebra.

Now we look at $\phi(M_n(\mathbb{F}_q))$, and in particular, its generators $%
\phi(a),\phi(b)$. We can write $\phi(a)$ as the limit of a Cauchy sequence $%
a_1,a_2,\dots$ and $\phi(b)$ as the limit of a Cauchy sequence $b_1,b_2,\dots
$, with each $a_i,b_i$ in the direct limit of $M_{n_0}(\mathbb{F}_q)
\hookrightarrow M_{n_1}(\mathbb{F}_q) \hookrightarrow \dots$. Let $\delta>0$%
. Then as the operations $+,\cdot,d$ are continuous, we can find some $%
a_i,b_j$ such that 
\begin{equation*}
d(a_i^n,0),d(b_j^n,0),d(a_i,\phi(a)),d(b_j,%
\phi(b)),d(b_ja_i+a_i^{n-1}b_j^{n-1},1)<\delta
\end{equation*}
and 
\begin{equation*}
\left| d(a_i,0)-\frac{n-1}{n} \right|,\left| d(b_j,0)-\frac{n-1}{n}
\right|<\delta.
\end{equation*}

Let $N$ be such that $a_i,b_j$ are both in the image $\iota_{n_N}(M_{n_N}(%
\mathbb{F}_q))$. Let $K\geq N$ and note that $a_i,b_j$ are both in the image 
$\iota_{n_K}(M_{n_K}(\mathbb{F}_q))$. Let $x=\iota_{n_K}^{-1}(a_i)$ and $%
y=\iota_{n_K}^{-1}(b_j)$ be matrices acting on the vector space $\mathbb{F}%
_q^{n_K}$. Then let $W_0=\ker y \cap \ker x^n \cap \ker (yx+x^{n-1}y^{n-1}-1)
$. Clearly $\dim \ker y=n_K-\mathrm{rank}(y)$, so $\left| \frac{\dim \ker y}{%
n_K} - \frac{1}{n} \right|<\delta$, and as each of the other operators has
rank at most $n_K\delta$, each of their kernels has dimension at least $%
n_K(1-\delta)$, and $W_0$, the intersection of these three kernels, has
dimension at least $n_K(\frac{1}{n}-3\delta)$.

As $W_0 \subset \ker (yx+x^{n-1}y^{n-1}-1) \,\cap\, \ker y$, we know that
for any $v \in W_0$, 
\begin{equation*}
(yx+x^{n-1}y^{n-1})v=yxv=v
\end{equation*}
so on the restricted domain of $W_0$, $yx=1$. Thus $x$ maps $W_0$
isomorphically onto $xW_0$ with inverse $y$. For any $1 \leq k \leq n-1$, if 
$W_{k-1} \subset \ker(yx-1) \cap \ker y^k$ has been defined, let $%
W_k=xW_{k-1} \cap \ker (yx+x^{n-1}y^{n-1}-1)$. Then as $yx$ is the identity
on $W_{k-1}$, $x$ is an isomorphism from $W_{k-1}$ to $xW_{k-1}$ with
inverse $y$. Thus $y$ is an isomorphism from $W_k$ to $yW_k \subset W_{k-1}
\subset \ker y^k$, so $W_k \subset \ker y^{k+1}$. If $k<n-1$, we also have $%
0=yx+x^{n-1}y^{n-1}-1=yx-1$ on $W_k$, so $W_k \in \ker (yx-1) \cap \ker
y^{k+1}$, satisfying the inductive hypothesis for the next step. Thus we can
apply this recursive definition all the way through $W_{n-1}$, provided the
base case works. As $W_0 \subset \ker (yx-1) \cap \ker y^1$, the base case
checks out, and this recursive definition is well-defined.

Also, as $\dim xW_{k-1}=\dim W_{k-1}$ and $yx+x^{n-1}y^{n-1}-1$ has rank at
most $n_K\delta$, 
\begin{equation*}
\dim W_{k}\geq \dim xW_{k-1} - n_K\delta =\dim W_{k-1}-n_K\delta.
\end{equation*}
As $\dim W_0 \geq n_K(\frac{1}{n}-4\delta)$, we have that $\dim W_k \geq n_K(%
\frac{1}{n}-(4+k)\delta) \geq n_K(\frac{1}{n}-(4+n)\delta)$. Now we define $%
V=W_{n-1}+ yW_{n-1}+ \dots + y^{n-1}W_{n-1}$, and we wish to prove that $V
\subset \ker x^n \cap \ker y^n \cap \ker (yx+x^{n-1}y^{n-1}-1)$. For any $0
\leq k \leq n-1$, as $y^{n-1}W_{n-1} \subset W_0 \subset \ker y$, $%
y^{k}W_{n-1} \subset \ker y^{n-k}$, so $V \subset \ker y^n$. For any $0 \leq
k \leq n-1$, $y^k W_{n-1} \subset x^{n-k-1}W_0$. As $W_0 \subset \ker x^n$,
we know that $y^k W_{n-1} \subset x^{n-k-1} W_0 \subset \ker x^{k+1} \subset
\ker x^n$, and $V \subset \ker x^n$. Thus also $W_{n-1} \subset \ker x
\subset \ker yx$, and as $xy$ is the identity on $W_k$ for any $k>0$, $%
x^ky=x^{k-1}$ on $W_{k}$, which contains $y^{k-1}W_{n-1}$, so as $xy$ is the
identity on $W_1$, by induction, $x^ky^k$ is the identity on $%
y^{n-k-1}W_{n-1}$, and $x^{n-1}y^{n-1}$ is the identity on $W_{n-1}$. Thus $%
yx+x^{n-1}y^{n-1}=1$ on $W_{n-1}$, and $W_{n-1} \subset \ker
(yx+x^{n-1}y^{n-1}-1)$. We now need to show that $y^{k}W_{n-1} \subset \ker
(yx+x^{n-1}y^{n-1}-1)$ for $k\geq 1$. In that case, $y^{k}W_{n-1} \subset
\ker y^{n-k} \subset y^{n-1}$, and by assumption, $y^kW_{n-1} \subset
W_{n-k-1} \subset \ker (yx-1)$, so $yx+x^{n-1}y^{n-1}-1=yx-1=0$ on $%
y^kW_{n-1}$ as desired.

Thus $V \subset \ker x^n \cap \ker y^n \cap \ker (yx+x^{n-1}y^{n-1}-1)$, and
the relations $x^n=y^n=0=yx+x^{n-1}y^{n-1}-1$ are satisfied when restricted
to the domain $V$. Thus on the domain $V$, for some unit $B \in M_{n_K}(%
\mathbb{F}_q)$ representing a change-of-basis, $x=B(a^{\oplus \frac{\dim V}{n%
}}\oplus 0^{\oplus n_K-\dim V})B^{-1}$ and $y=B(b^{\oplus \frac{\dim V}{n}%
}\oplus 0^{\oplus n_K-\dim V})B^{-1}$.

Let $n_K=nm+r$, with $0\leq r<n$. Then as $n$ divides $\dim V$, $\frac{\dim V%
}{n}\leq m$. If we let $x^{\prime \oplus m}\oplus 0^{\oplus r})B^{-1}$ and $%
y=B(b^{\oplus m}\oplus 0^{\oplus r})B^{-1}$, we see that $x^{\prime \oplus 
\frac{\dim V}{n}}a^{\oplus m-\frac{\dim V}{n}}\oplus 0^{\oplus r})B^{-1}$
and $y^{\prime \oplus \frac{\dim V}{n}}b^{\oplus m-\frac{\dim V}{n}}\oplus
0^{\oplus r})B^{-1}$, so both of these clearly have rank 
\begin{equation*}
m-\frac{\dim V}{n}\leq \frac{n_K-\dim V}{n}
\end{equation*}
Thus if we define the homomorphism $\psi$ on the generators by $%
\psi(a)=x^{\prime }$ and $\psi(b)=y^{\prime }$, we find that $\psi$ is a $%
\frac{r}{n_K}$-embedding, and it suffices to show that $d(\iota_{n_K}\circ
\psi,\phi)<\varepsilon$. To do this, it suffices to show that $%
d(\iota_{n_K}\circ \psi(a),\phi(a)),d(\iota_{n_K}\circ
\psi(b),\phi(b))<\gamma$ for some $\gamma<0$ depending on $\varepsilon$. As 
\begin{equation*}
d(\iota_{n_K}\circ \psi(a),\phi(a))=d(\iota_{n_K}(x^{\prime }),\phi(a))\leq
d(x,x^{\prime })+d(\iota_{n_K}(x),\phi(a))\leq d(x,x^{\prime })+\delta,
\end{equation*}
and similarly $d(\iota_{n_K}\circ \psi(b),\phi(b))\leq d(y,y^{\prime
})+\delta$, we recall that $d(x,x^{\prime }),d(y,y^{\prime })\leq \frac{%
n_K-\dim V}{n}$, so we only need to show that $\frac{n_K-\dim V}{n}%
+\delta<\gamma$ for sufficiently small $\delta$.

For $0\leq r<s \leq n$, we will show that $y^r W_{n-1} \cap y^s W_{n-1}=0$.
As $y^r W_{n-1} \subset \ker x^{r+1}$ and $y^s W_{n-1} \subset \ker y^{n-s}$%
, if $v \in y^r W_{n-1} \cap y^s W_{n-1}$, then $v \in \ker x^{r+1} \cap
y^{n-s}$. As $\langle x,y\rangle \cong M_n(\mathbb{F}_q)$ when restricted to
the domain $V$, $x^sy^s+y^{n-s}x^{n-s}=1$ on $V$, so $%
v=(y^sx^s+x^{n-s}y^{n-s})v=0$. Thus $y^r W_{n-1} \cap y^s W_{n-1}=0$, and 
\begin{equation*}
\dim V=\sum_{s=0}^{n-1} \dim y^s W_{n-1}=n \dim W_{n-1}\geq
n_K(1-(4+n)n\delta)
\end{equation*}
Placing this in our earlier inequality, we find that 
\begin{equation*}
d(x,x^{\prime }),d(y,y^{\prime })<\frac{n_K-\dim V}{n_K} \leq (4+n)n\delta
\end{equation*}
so by taking $\delta$ low enough, we find 
\begin{equation*}
d(x,x^{\prime })+\delta,d(y,y^{\prime 2}+4n+1)\delta<\gamma
\end{equation*}
as desired.
\end{proof}

\section{Explicit Fra\"iss\'e Theory}

\label{sec:limit} 
As $\mathcal{K}(\mathbb{F}_q)$ is a Fra\"iss\'e class, it must have a unique
Fra\"iss\'e limit. A Fra\"iss\'e limit of $\mathcal{K}(\mathbb{F}_q)$ is a 
\textit{$\mathcal{K}(\mathbb{F}_q)$-structure} which is \textit{$\mathcal{K}(%
\mathbb{F}_q)$-universal} and \textit{approximately homogeneous} \cite{Eagle}%
. A $\mathcal{K}(\mathbb{F}_q)$-structure is a structure which can be
realized as the direct limit (in the category of metric structures with the
appropriate signature) of a sequence of elements of $\mathcal{K}(\mathbb{F}%
_q)$, which is, in this case, the completion of the algebraic direct limit
of an inductive sequence of elements of $\mathcal{K}(\mathbb{F}_q)$, or $M(%
\mathbb{F}_q)$ for some factor sequence. Given von Neumann's result, it is
clear that there is only one $\mathcal{K}(\mathbb{F}_q)$-structure up to
isomorphism, so this must be the Fra\"iss\'e limit. If we do not assume this
result, we can still use the uniqueness of the Fra\"iss\'e limit to directly
prove the uniqueness of $M(\mathbb{F}_q)$.

First we will show that if the factor sequence $n_0|n_1|\dots$ is given by $%
n_i=i!$, then the completion $M(\mathbb{F}_q)$ of the corresponding direct
limit is a Fra\"iss\'e limit. Then we will show, with a back-and-forth
argument mirroring the classic proof of the uniqueness of the Fra\"iss\'e
limit, that all $\mathcal{K}(\mathbb{F}_q)$-structures are isomorphic to $M(%
\mathbb{F}_q)$.\cite{Hodges}

The factor sequence $0!,1!,2!,\dots$ is chosen to make $\mathcal{K}$%
-universality simple to prove. For any $i$, we know an embedding of $M_i(%
\mathbb{F}_q)$ into $M_{i!}(\mathbb{F}_q)$, as $i$ divides $i!$. Thus $M_i(%
\mathbb{F}_q)$ embeds into $M_{n_i}(\mathbb{F}_q)$, and thus into $M(\mathbb{%
F}_q)$.

\subsection{Approximate Homogeneity}

Fix a factor sequence $n_0|n_1|\dots$. We will show that the completion of
its direct limit, $M(\mathbb{F}_q)$, is the Fra\"iss\'e limit of all $M_n(%
\mathbb{F}_q)$s. Let $\phi,\psi:M_n(\mathbb{F}_q) \hookrightarrow M(\mathbb{F%
}_q)$ be embeddings. Fix $\varepsilon>0$. Then we apply Lemma \ref{stablem}
to both $\phi,\psi$, letting $N_\phi$ be the value of $N$ that suffices for $%
\phi$, and $N_\psi$ the value of $N$ that suffices for $\psi$. Then let $%
K=\max{N_\phi,N_\psi}$. By the choice of $N_\phi$ and $N_\psi$, we see that
there is a ring homomorphism $\phi^{\prime }: M_n(\mathbb{F}_q) \to M_{n_K}(%
\mathbb{F}_q)$ with $d(\iota_{n_K} \circ \phi^{\prime },\phi)<\frac{%
\varepsilon}{2}$, and similarly a homomorphism $\psi^{\prime }$ close to $%
\psi$, together with units $B_\phi,B_\psi \in M_{n_K}(\mathbb{F}_q)$ such
that where $n_K=nm+r$ and $0 \leq r<n$, for all $A \in M_n(\mathbb{F}_q)$, $%
\phi^{\prime }(A)=B_\phi(A^{\oplus m} \oplus 0^{\oplus r})B_\phi^{-1}$, and
similarly for $\psi^{\prime }$. Thus $B_\psi B_\phi^{-1}\phi^{\prime
}(A)B_\phi B_\psi^{-1}=B_\psi(A^{\oplus m} \oplus 0^{\oplus
r})B_\psi^{-1}=\psi^{\prime }(A)$. Thus if $\beta$ is the inner automorphism
given by conjugation by $B_\psi B_\phi^{-1}$, we have $\psi^{\prime }=\beta
\circ \phi^{\prime }$, and for $A \in M_n(\mathbb{F}_q)$, 
\begin{align*}
d(\beta \circ \phi(A),\psi(A)) &\leq d(\beta \circ \phi(A),\beta \circ
\phi^{\prime }(A))+d(\psi(A),\psi^{\prime }(A)) \\
&=d(\phi(A),\phi^{\prime }(A))+d(\psi(A),\psi^{\prime }(A)) \\
&<\varepsilon
\end{align*}
so $d(\beta \circ \phi,\psi)<\varepsilon$.

\subsection{Uniqueness}

The Fra\"iss\'e limit is also unique, at least among $\mathcal{K}(\mathbb{F}%
_q)$-structures, which are direct limits of elements of $\mathcal{K}(\mathbb{%
F}_q)$. \cite{Eagle} In fact, as von Neumann and Halperin showed, there is
only one $\mathcal{K}(\mathbb{F}_q)$-structure, $M(\mathbb{F}_q)$ (as metric
structures must be complete, the direct limit of an inductive sequence of
metric $\mathbb{F}_q$-algebras is the completion of the algebraic direct
limit). We shall provide an alternate proof of this fact, following the
classic proof of the uniqueness of the Fra\"iss\'e limit.

First we show the approximate extension property.

\begin{lem}[Approximate Extension Property]
\label{extension}  Let $m_0,m_1,\dots$ be a factor sequence. Fix $%
\delta,\delta^{\prime }>0$, and let $\phi:M_{m_k}(\mathbb{F}_q)\to M_n(%
\mathbb{F}_q)$ be a $\delta$-embedding. There exists some $k^{\prime }\geq k$
and a $\delta^{\prime }$-embedding $\psi: M_{n}(\mathbb{F}_q) \to
M_{m_{k^{\prime }}}(\mathbb{F}_q)$ such that the following diagram commutes
up to $\delta+\delta^{\prime }$:  
\begin{equation*}
\begin{tikzcd}
		M_{m_{k'}}(\mathbb{F}_q) & M_n(\mathbb{F}_q) \lar["\psi" above]\\
		M_{m_k}(\mathbb{F}_q)\uar[ hook, "i_{m_{k'},m_k}"]\urar["\phi" below] 
	\end{tikzcd}
\end{equation*}
\end{lem}

\begin{proof}
As $\phi$ is a $\delta$-embedding, we can write it as $\phi: a \mapsto
y(a^{\oplus r} \oplus 0^{n-m_kr})y^{-1}$ where $m_kr\geq (1-\delta)n$. Let $%
\psi: b \mapsto z((y^{-1}by)^{\oplus s}\oplus 0^{m_{k^{\prime }}-sn})z^{-1}$%
, which makes $\psi$ a $\delta^{\prime }$-embedding as long as $\frac{%
m_{k^{\prime }}-sn}{m_{k^{\prime }}}\leq \delta^{\prime }$ which is
satisfied when $\delta^{\prime }m_{k^{\prime }}>n$, and $s=\lfloor \frac{%
m_{k^{\prime }}}{n}\rfloor$, so $m_{k^{\prime }}-sn<n<\delta^{\prime
}m_{k^{\prime }}$.

Then the only remaining requirement is that for all $a \in M_{m_k}(\mathbb{F}%
_q)$, $d(\psi \circ \phi(a),i_{m_{k^{\prime
}},m_k}(a))>\delta+\delta^{\prime }$.  
\begin{align*}
\psi \circ \phi(a) &=z((y^{-1}y(a^{\oplus r} \oplus
0^{n-m_kr})y^{-1}y)^{\oplus s}\oplus 0^{m_{k^{\prime }}-sn})z^{-1} \\
&\leq z((a^{\oplus r} \oplus 0^{n-m_kr})^{\oplus s}\oplus 0^{m_{k^{\prime
}}-sn})z^{-1}
\end{align*}
so with the correct choice of $z$,  
\begin{equation*}
\psi \circ \phi(a)=a^{\oplus rs}\oplus 0^{m_{k^{\prime }}-rsm_k}
\end{equation*}
so  
\begin{equation*}
d(\psi \circ \phi(a),i_{m_{k^{\prime }},m_k}(a))=1-\frac{rsm_k}{m_{k^{\prime
}}}
\end{equation*}
and  
\begin{equation*}
\frac{rsm_k}{m_{k^{\prime }}}\geq (1-\delta)\frac{sn}{m_{k^{\prime }}}\geq
(1-\delta)(1-\delta^{\prime })>1-(\delta+\delta^{\prime })
\end{equation*}
so we have  
\begin{equation*}
d(\psi \circ \phi(a),i_{m_{k^{\prime }},m_k}(a))=1-\frac{rsm_k}{m_{k^{\prime
}}}<\delta+\delta^{\prime }
\end{equation*}
as desired.
\end{proof}

\begin{thm}
If $X$ and $Y$ are $\mathcal{K}(\mathbb{F}_q)$-structures, then $X \cong Y$.
\end{thm}

\begin{proof}
Let $X$ be the completion of a direct limit corresponding to the from the
factor sequence $m_0,m_1,\dots$, and let $Y$ be the completion of the direct
limit corresponding to the factor sequence $n_0,n_1,\dots$. Let $\phi_0:
M_{m_0}(\mathbb{F}_q) \to M_{n_0}(\mathbb{F}_q)$ be a $1$-embedding, so $%
i_{n_1,n_0}\circ \phi_0$ is a $1$-embedding as well. Thus by Lemma \ref%
{extension}, there is some $M_{m_{j_1}}(\mathbb{F}_q)$ and a $2^{-1}$%
-embedding $\psi_0: M_{n_0}(\mathbb{F}_q) \to M_{m_{j_1}}$ such that $%
d(\psi_0 \circ i_{n_1,n_0}\circ \phi_0,i_{m_{j_1},m_0})<1+2^{-1}$. We define 
$j_0=k_0=0$, and given $j_i$ or $k_i$, define $X_i=M_{m_{j_i}}$ and $%
Y_i=M_{n_{k_i}}$. Now we continue this process recursively. Let $%
\phi_i:X_{2i} \to Y_{2i}$ is a $2^{-2i}$-embedding. We define $%
k_{2i+1}=k_{2i}+1$, so that the sequence $Y_0,Y_1,\dots$ does not terminate.
By Lemma \ref{extension}, there is a $X_{2i+1}$ and a $2^{-(2i+1)}$%
-embedding $\psi_i: Y_{2i+1} \to X_{2i+1}$ such that $d(\psi_i \circ
\iota_{n_{k_{2i+1}},n_{k_{2i}}}\circ
\phi_i,\iota_{m_{j_{2i+1}},m_{j_{2i}}})<2^{-2i}+2^{-(2i+1)}$. This just
generalizes the case of $i=0$.

Similarly, let $\psi_i:Y_{2i+1} \to X_{2i+1}$ be a $2^{-(2i+1)}$-embedding.
Then we define $j_{2i+2}=j_{2i+1}+1$, so that the sequence $X_0,X_1,\dots$
does not terminate either. By Lemma \ref{extension}, there is a $Y_{2i+2}$
and a $2^{-(2i+2)}$-embedding $\phi_{i+1}: X_{2i+2} \to Y_{2i+2}$ such that $%
d(\phi_{i+1} \circ \psi_i,\mathrm{id})<2^{-(2i+1)}+2^{-(2i+2)}$.

We will show that $d(\phi_{n}(x),\phi_{n+1}(x))<2^{-n+2}$. 
\begin{align*}
d(\phi_n(x),\phi_{n+1}&(x)) \\
&\leq d(\phi_n(x),\phi_{n+1}\circ \psi_n \circ \phi_n(x))+d(\phi_{n+1}\circ
\psi_n \circ \phi_n(x),\phi_{n+1}(x)) \\
&\leq d(\phi_n(x),\phi_{n+1}\circ \psi_n \circ \phi_n(x))+d(\psi_n \circ
\phi_n(x),x) \\
&<(2^{-2n}+2^{-2n-1})+(2^{-2n-2}+2^{-2n-3})<2^{-2n+1}
\end{align*}
and thus if $n<m$, $d(\phi_m(x),\phi_n(x))<%
\sum_{k=n}^{m-1}2^{-2k+1}<2^{-2n+2}$, and $\phi_n(x),\phi_{n+1}(x),\dots$ is
a Cauchy Sequence, and by the same proof so is $\psi_n(y),\dots$. Now we
define $\phi=\lim_i \phi_i$ and $\psi=\lim_i \psi_i$ pointwise on $\bigcup_i
X_i$ and $\bigcup_i Y_i$ respectively. Then for any $i$ and $x \in X_{2i}$, $%
d(\psi_i \circ \phi_i(x),x)<2^{-2i}+2^{-2i-1}$. Thus if $x \in X_{2i}$, $%
\lim_{j \to \infty} d(\psi_j \circ \phi_j(x),x)=0$.

For any $\delta$-embedding $\theta:A \to B$, $d(\theta(1_A),1_B)<\delta$, so 
$\phi(1)=\lim_i \phi_i(1)=1$, and by the same reasoning, $\psi$ is also
unital.

Now we show that $\psi,\phi$ are 1-Lipschitz. Fix $\varepsilon>0$, and let $%
x_1,x_2 \in \bigcup_i X_i$ be such that $d(x_1,x_2)<\varepsilon$. Let $j$ be
such that $x_1,x_2 \in X_{2n}$. Then if $r(x)$ is the normalized rank of $x$%
, we have for any $x \in X_{2n}$, $r(\phi_{n}(x))\leq r(x)$ because $\phi_j$
is a $\delta$-embedding for some $\delta$, and thus $d(\phi_{n}(x_1),%
\phi_{n}(x_2))\leq d(x_1,x_2)<\varepsilon$, so $\phi$ is 1-Lipschitz on $%
\bigcup_i X_i$, which is dense in $X$, so it is possible to extend $\phi$ to
the entirety of $X$ as a 1-Lipschitz and thus continuous map. Similarly, $%
\psi$ is 1-Lipschitz, and can be extended to all of $Y$.

Fix $x_1,x_2 \in \bigcup_i X_i$, and any $\delta>0$, there is some $N$ such
that if $n>N$, $\phi_n$ is a $\delta$-embedding. Thus also $%
d(\phi_n(x_1),\phi_n(x_2))\geq (1-\delta)d(x_1,x_2)$, so $%
d(\phi(x_1),\phi(x_2))=d(x_1,x_2)$, and $\phi$ (and similarly $\psi$) is an
isometry.

We now wish to show that $\psi \circ \phi$ is the identity on $X$, and the
same proof will show that $\phi \circ \psi$ is the identity on $Y$. We note
that for any $n<m$, as proven earlier, $d(\phi_m(x),\phi_n(x))<2^{-n+2}$,
and $d(\psi_m(y),\psi_n(y))<2^{-2n+1}$. Thus for any $n$, 
\begin{align*}
d(\psi \circ &\phi(x),x) \\
&\leq d(\psi \circ \phi(x), \psi \circ \phi_n(x))+d(\psi \circ
\phi_n(x),\psi_n \circ \phi_n(x))+d(\psi_n\circ \phi_n(x),x) \\
&< d(\phi(x),\phi_n(x))+d(\psi \circ \phi_n(x),\psi_n \circ
\phi_n(x))+(2^{-2n}+2^{-2n-1}) \\
&< 2^{-2n+2}+2^{-2n+1}+2^{-2n}+2^{-2n-1} \\
&<2^{-2n+3}
\end{align*}
and therefore $d(\psi \circ \phi(x),x)=0$, as desired.

As $\psi$ and $\phi$ are inverses, and each is a unital isometric
homomorphism, they are isomorphisms, and $X \cong Y$.
\end{proof}

\section{Extreme Amenability}

\label{sec:exam}

\subsection{The set of inner automorphisms is dense in $\mathrm{Aut}(M(%
\mathbb{F}_q))$}

In order to show that the inner automorphisms are dense in $\mathrm{Aut}(M(%
\mathbb{F}_q))$, it suffices to choose an automorphism $\phi \in \mathrm{Aut}%
(M(\mathbb{F}_q))$, and a basis open neighborhood around it, and find an
inner automorphism in that neighborhood. Fix a factor sequence $n_0,n_1,\dots
$, and let $M_0(\mathbb{F}_q)$ be the direct limit associated to it, dense
in $M(\mathbb{F}_q)$. Now let $U$ be a basis neighborhood around $\phi$,
which will be of the form $\bigcap_{x \in X}\{f \in \mathrm{Aut}(M(\mathbb{F}%
_q)): d(f(x),\phi(x))<\varepsilon\}$ for some finite set $X=\{x_1,\dots,x_k\}
$.

For each $x_i \in X$, let $y_i \in M_0(\mathbb{F}_q)$ be such that $%
d(x_i,y_i)<\frac{\varepsilon}{3}$. Then we shall find an inner automorphism $%
\psi$ such that $d(\psi(y_i),\phi(y_i))<\frac{\varepsilon}{3}$ for each $y_i$%
. Given such an automorphism, we find 
\begin{align*}
d(\psi(x_i),\phi(x_i))&<d(\psi(x_i),\psi(y_i))+d(\psi(y_i),\phi(y_i))+d(%
\phi(y_i),\phi(x_i)) \\
&<2d(x_i,y_i)+\frac{\varepsilon}{3} \\
&<\varepsilon
\end{align*}
so $\psi \in U$. Thus it suffices to find an inner automorphism $\psi$ such
that $d(\psi(y),\phi(y))<\varepsilon$ for each $y \in Y$, for each finite
set $Y \subset M_0(\mathbb{F}_q)$ and each $\varepsilon>0$.

Let us fix some such $Y \subset M_0(\mathbb{F}_q)$ and $\varepsilon>0$. As $Y
$ is finite, there must be some $n_m$ such that $Y$ is contained in the
image $\iota_{n_m}(M_{n_m}(\mathbb{F}_q))$.

Define $\phi^{\prime }: M_{n_m}(\mathbb{F}_q) \hookrightarrow M(\mathbb{F}_q)
$ by $\phi^{\prime }=\phi \circ \iota_{n_m}$. Clearly $\phi^{\prime }$ is an
embedding. By Lemma \ref{stablem}, there exists some $n_K$, and an embedding 
$\psi: M_{n_m}(\mathbb{F}_q)\hookrightarrow M_{n_K}(\mathbb{F}_q)$ such that 
$d(\iota_{n_K} \circ \psi(x),\phi^{\prime }(x))<\varepsilon$ for each $x \in
M_{n_m}(\mathbb{F}_q)$, so for each $y \in Y$, as $Y \subset M_{n_m}(\mathbb{%
F}_q)$, $d(\iota_{n_K} \circ \psi(y),\phi^{\prime }(y))<\varepsilon$.

\subsection{Quotient}

Let $A(p)$ be as in the paper by Carderi and Thom.\cite{carderithom} They
assert that $A(p)$ is the group of units of $M(\mathbb{F}_q)$, which
obviously has a continuous surjective homomorphism onto the inner
automorphism group $B(p)$ of units of $M(\mathbb{F}_q)$. As $B(p)$, the
image of $A(p)$ under a continuous homomorphism, is dense in $\mathrm{Aut}(M(%
\mathbb{F}_q))$, we have, with Proposition 6.2 of \cite{hereditary}, that $%
\mathrm{Aut}(M(\mathbb{F}_q))$ is itself extremely amenable, and by the KPT
correspondence, the class of matrix algebras has the Ramsey Property.

\section{Ramsey Property}

\label{sec:ramsey}

\begin{thm}
The Fra\"iss\'e class $\mathcal{K}(\mathbb{F}_q)$ has the approximate Ramsey
Property. That is, if $A,B \in \mathcal{K}(\mathbb{F}_q)$, and $\varepsilon>0
$, there exists some $C \in \mathcal{K}(\mathbb{F}_q)$ such that for any
continuous coloring $\gamma$ of ${\binom{C }{A}}$, that is, a 1-Lipschitz
map $\gamma: {\binom{C }{A}} \to [0,1]$, there is some $B^{\prime }\in {%
\binom{C }{B}}$ such that the oscillation of $\gamma$ over the subset ${%
\binom{B^{\prime }}{A}} \subset {\binom{C }{A}}$ is at most $\varepsilon$.
\end{thm}

\begin{proof}
Let $A,B \in \mathcal{K}(\mathbb{F}_q)$ be such that $A \leq B$, that is, $%
A=M_a(\mathbb{F}_q), B=M_b(\mathbb{F}_q)$, where $a$ divides $b$. Let $%
k=\left| {\binom{B }{A}} \right|$. Fix $\varepsilon$.

Let $c>64
\varepsilon^{-2}\max\{\log(2k),\log(6\lceil\varepsilon^{-1}\rceil)\}$ be a
multiple of $b$, and $C=M_c(\mathbb{F}_q)$. Now let $\gamma$ be a continuous
coloring of ${\binom{C }{A}}$. We seek to find some $B^{\prime }\leq C$ with 
$B \cong B^{\prime }$ such that the oscillation of $\gamma$ on ${\binom{%
B^{\prime }}{A}}$ is at most $\varepsilon$.

Let ${\binom{B }{A}}=\{A_1,\dots,A_k\}$. For each $A_j \in {\binom{B }{A}}$,
there exists some inner automorphism $\phi_j$ of $B$ such that $%
\phi_j(A)=f_j A f_j^{-1}=A_j$, with $f_j$ a unit of $B$, which can be taken
to have determinant 1, so that $f_j \in SL_b(\mathbb{F}_q)$. Let $F=\{f_j
\otimes 1_{c/b}: 1 \leq j \leq k\}$, so that $F \subset SL_c(\mathbb{F}_q)$.
Now we define a coloring $\gamma^{\prime }$ of $SL_c(\mathbb{F}_q)$, given
by $\gamma^{\prime -1})$.

Now let $m=\left\lceil \frac{3}{\varepsilon}\right\rceil$, and define $%
\mathcal{U}$ to be an open cover $\{U_i: 1 \leq i \leq m\}$ of $SL_c(\mathbb{%
F}_q)$, such that every $\frac{\varepsilon}{3}$-ball in $SL_c(\mathbb{F}_q)$
is contained in some $U_i$. Specifically, let us let $V_i=\left(\frac{i-2}{3}%
\varepsilon,\frac{i+1}{3}\varepsilon\right)$, observing that $\{V_i: 1 \leq
i \leq m\}$ is an open cover for $[0,1]$. We note that the $\frac{\varepsilon%
}{3}$-ball around any point in $[0,1]$ is contained in some $V_i$. Then let $%
U_i=\gamma^{\prime -1}(V_i)$. As $\gamma^{\prime }$ is 1-Lipschitz, if $x
\in SL_c(\mathbb{F}_q)$, and $B(x)$ is the $\frac{\varepsilon}{3}$-ball
around it, then for any $y \in B(x)$, $d(\gamma^{\prime }(y),\gamma^{\prime
}(x))\leq d(y,x)<\varepsilon$, so $\gamma^{\prime }(B(x))$ is contained in $%
B(\gamma^{\prime }(x))$, which is contained in turn by $V_i$ for some $i$,
so $B(x) \subset \gamma^{\prime -1}(V_i)=U_i$.

Theorem 2.8 of Carderi and Thom's paper\cite{carderithom} states that there
exists some $g \in SL_c(\mathbb{F}_q)$ such that $gF \subset U_i$ for some $%
1\leq i\leq m$, as long as we take $c > 64
\varepsilon^{-2}\max\{\log(2k),\log(2m)\}$, which is satisfied by our choice 
\begin{equation*}
c>64 \varepsilon^{-2}\max\{\log(2k),\log(6\lceil\varepsilon^{-1}\rceil)\}
\end{equation*}
Thus for each $f_j \in F$, $\gamma^{\prime }(gf_j) \in V_i$, so $%
\gamma(gf_jAf_j^{-1}g^{-1}) \in V_i$. As $f_j A f_j^{-1} \subset B$, we have 
$gf_j A f_j^{-1}g^{-1} \subset gBg^{-1}$, and thus $gf_j A f_j^{-1}g^{-1}
\in {\binom{gBg^{-1} }{A}}$. Thus if $S=\{gf_j A f_j^{-1} g^{-1}: 1 \leq j
\leq k\}$, then $S \subset {\binom{gBg^{-1} }{A}}$, and $\gamma(S) \in V_i$,
so the oscillation of $\gamma$ on $S$ is at most the diameter of $V_i$,
which is $\varepsilon$. Let $1\leq j,\ell\leq k$. If $%
gf_jAf_j^{-1}g^{-1}=gf_\ell Af_\ell ^{-1}g^{-1}$, then $f_j Af_j
^{-1}=f_\ell Af_\ell ^{-1}$, so $j=\ell$. Thus each $gf_j A f_j^{-1} g^{-1}$
is distinct, and $\left| \{gf_j A f_j^{-1} g^{-1}: 1 \leq j \leq k\}
\right|=k=\left| {\binom{gBg^{-1} }{A}} \right|$, so $\{gf_j A f_j^{-1}
g^{-1}: 1 \leq j \leq k\}={\binom{gBg^{-1} }{A}}$, and ${\binom{gBg^{-1} }{A}%
}$ has oscillation at most $\varepsilon$ under $\gamma$, so we can let $%
B^{\prime -1}$.
\end{proof}

Now to make precise our bound of $64
\varepsilon^{-2}\max\{\log(2k),\log(6\lceil\varepsilon^{-1}\rceil)\}$, we
must bound $k$. If $A^{\prime }\in {\binom{B }{A}}$, and $\phi \in \mathrm{%
Aut}(B)$, then $\phi(A^{\prime })$ will still be an embedded copy of $A$,
and it is easy to see that $\phi: A^{\prime }\mapsto \phi(A^{\prime })$
defines a group action of $\mathrm{Aut}(B)$ on ${\binom{B }{A}}$. For any $%
A^{\prime }\in {\binom{B }{A}}$, by the Skolem-Noether theorem there is some
automorphism $\phi \in \mathrm{Aut}(B)$ such that $\phi(A \otimes
1)=A^{\prime }$, so this action is transitive, and thus $k=\left| {\binom{B 
}{A}} \right|=\frac{\left| \mathrm{Aut}(B) \right|}{\left| \mathrm{Stab}(A
\otimes 1) \right|} \leq \left| \mathrm{Aut}(B) \right|$. Thus we calculate $%
\mathrm{Aut}(B)$.

\begin{thm}
$\mathrm{Aut}(M_n(\mathbb{F}_q))\cong SL_n(\mathbb{F}_q)$
\end{thm}

\begin{proof}
Define a map $\phi: GL_n(\mathbb{F}_q) \to \mathrm{Aut}(M_n(\mathbb{F}_q))$
given by $\phi(g):x \mapsto gxg^{-1}$. Not only is this a group
homomorphism, but it is onto, as the Skolem-Noether theorem guarantees that
all automorphisms of $M_n(\mathbb{F}_q)$ are inner. Thus it suffices to
determine the kernel of the map. $\ker \phi$ is exactly the center of $GL_n(%
\mathbb{F}_q)$, which is just the scalar multiples of the identity. Thus $%
\mathrm{Aut}(M_n(\mathbb{F}_q))\cong GL_n(\mathbb{F}_q)/GL_1(\mathbb{F}_q)
\cong SL_n(\mathbb{F}_q)$.
\end{proof}

As $\left| SL_n(\mathbb{F}_q) \right|=\frac{1}{q-1}\prod_{i=0}^{n-1}(q^n-q^i)
$\cite{sln}, $k \leq q^{b^2}$, so our Ramsey bound can be written as 
\begin{equation*}
64\varepsilon^{-2}(\log(2)+\mathrm{max}(b^2\log(q),\log(6\lceil
\varepsilon^{-1} \rceil))
\end{equation*}
which is remarkably only quadratically dependent in $b$, and only slightly
worse than quadratic in $\varepsilon^{-1}$.

\providecommand{\bysame}{\leavevmode\hbox to3em{\hrulefill}\thinspace} %
\providecommand{\MR}{\relax\ifhmode\unskip\space\fi MR } 
\providecommand{\MRhref}[2]{  \href{http://www.ams.org/mathscinet-getitem?mr=#1}{#2}
} \providecommand{\href}[2]{#2}

\end{document}